\documentclass[11pt]{amsart}
\usepackage{enumerate}
\usepackage{enumitem}
\usepackage{marginnote}
\usepackage{amsmath}
\usepackage{amssymb, pb-diagram}
\usepackage[all]{xy}
\usepackage{graphicx} 
\usepackage{mathabx}
\usepackage{chapterbib}
\usepackage{epsfig}
\usepackage{amsfonts}
\usepackage{amssymb}
\usepackage{amsmath}   
\usepackage{amsthm}
\usepackage{color}
\usepackage{latexsym}
\usepackage{hyperref}
\usepackage[a4paper,top=2cm,bottom=2cm,left=3.5cm,right=3.5cm, marginparwidth=3cm, marginparsep=3mm]{geometry}

\def\dst{\displaystyle}

\newtheorem{thm}{Theorem}[section]
\newtheorem*{thmi}{Main Theorem}

\newtheorem{prop}[thm]{Proposition}

\newtheorem{cor}[thm]{Corollary}
 \theoremstyle{definition}
  \newtheorem{definition}[thm]{Definition}
    
 \newtheorem{question}[thm]{Question}
  \newtheorem{example}[thm]{Example}
  
    \newtheorem{rem}[thm]{Remark}

\DeclareMathOperator {\Hn}{\operatorname{H} }
\DeclareMathOperator {\cd}{\operatorname{cd} }
   \DeclareMathOperator {\gd}{\operatorname{gd} }

     \DeclareMathOperator {\X}{\mathfrak{X}}  
      
  \DeclareMathOperator{\pd}{pd}

      \DeclareMathOperator {\Ob}{\operatorname{obj}}

    \def\epi{\twoheadrightarrow}

     \DeclareMathOperator {\colim}{\operatorname{colim}}

 \DeclareMathOperator {\Mod}{\operatorname{Mod}} 
 \newcommand{\RMod}{R\,\mbox{-}\Mod}
 \newcommand{\DMod}{\mathfrak{D}\,\mbox{-}\Mod}

    \DeclareMathOperator{\Ab}{\mathcal Ab}

  \def\OCOG{{\mathcal O}_{\mathfrak{CO}}G}
  
  \def\OCG{{\mathcal O}_{\mathfrak C}G}
   \def\OCGp{{\mathcal O'}_{\mathfrak C}G}

     \def\OCO{{\mathcal O}_{\mathfrak{CO}}}
     
    \def\MCOG{{\mathcal M}_{(\FC,\FO)}G}
    
  \def\MCOOG{{\mathcal M}_{\FCO}G}
   \def\MCO{{\mathcal M}_{(\FC,\FO)}}
    
  \def\MCOO{{\mathcal M}_{\FCO}}

   \def\FC{\mathfrak C}
   \def\FK{\mathfrak K}
\def\FO{\mathfrak O}
\def\FCO{\mathfrak{CO}}
   \def\SCO{{}_G\mathcal S_{(\FC,\FO)}}
\def\SCOO{{}_G\mathcal S_{\FCO}}
\def\SCOF{{}_G\mathcal S_{(\FC,\FO)}^f}

\def\Ob{\operatorname{\bf Ob}}
\def\Mor{\operatorname{\bf Mor}}

\def\Hom{\operatorname{\rm Hom}}

\def\Ext{\operatorname{\rm Ext}}

\def\Z{\mathbb Z}
\def\N{\mathbb N}
\def\Q{\mathbb Q}

\def\O{\mathcal O}

\def\CS{\mathcal S}

\def\FC{\mathfrak C}
\def\D{\mathfrak{D}}
\def\FO{\mathfrak O}
\def\FF{\mathfrak F}
\def\FCO{\mathfrak{CO}}

\def\Msys{(\FC,\FO)}

\numberwithin{equation}{section}

\begin{document}

\title[T.d.l.c.~ mackey]{On cohomological dimensions of totally disconnected locally compact groups}
\author{Ilaria Castellano}
\address{Dipartimento di Matematica e Applicazioni, Universit\`a degli Studi di Milano-Bicocca, 20126, Milano}
\email{ilaria.castellano@unimib.it}
\author{Nadia Mazza}
\address{School of Mathematical Sciences, Lancaster University, Lancaster, LA1 4YF}
\email{n.mazza@lancaster.ac.uk}

\author{Brita Nucinkis}
\address{Department of Mathematics, Royal Holloway, University of London, Egham, TW20 0EX}
\email{Brita.Nucinkis@rhul.ac.uk}

\keywords{Mackey functors, t.d.l.c.~groups, cohomological dimension}

\begin{abstract} In this paper, we introduce Mackey functors for a t.d.l.c.~group and define the cohomological dimension of this group over the Mackey category. We then compare this dimension to the rational discrete cohomological dimension defined by Castellano and Weigel, as well as to the Bredon cohomological dimension of that {t.d.l.c.}~group with respect to the family of compact open subgroups. We also extend results about the geometric dimension of a t.d.l.c.~group. 

\end{abstract}
\subjclass{Primary 20J05, 22D05; Secondary 18Gxx}

\maketitle

\section{Introduction}

Totally disconnected locally compact (~= t.d.l.c.~) groups are a class of groups containing both the class of profinite groups and the class of discrete groups. For both of these two classes of groups there is a well-developed theory of Mackey functors beginning with finite groups~\cite{tw-structure}, infinite discrete groups~\cite{degrijse, mp-n, stjg},  and profinite groups~\cite{BB, ds}.  The list of citations is by no means complete, but serves as a starting point. In the case of discrete groups, the category of {Mackey functors} has enough projective objects~\cite{mp-n,tw-structure} and, for finite groups, the Burnside Mackey functor is projective. This is not the case for infinite discrete groups, where one considers Mackey functors evaluated at the family of finite subgroups. Hence, the Mackey cohomological dimension of a given group is defined to be the cohomological dimension of the Burnside Mackey functor. In~\cite{mp-n} the authors compared this dimension to the Bredon cohomological dimension, with respect to the family of finite subgroups, and to the rational cohomological dimension.

In this note, we shall extend these results to t.d.l.c.~groups and define Mackey functors in the most general setting via Mackey systems $(\FC, \FO)$ as defined in~\cite{BB}.  Here $\FC$ is a family of closed subgroups and $\FO$ a subfamily satisfying certain finiteness conditions, see Definition~\ref{def:Mackey sys}. We then concentrate on Mackey functors and Bredon cohomology with respect to the family $\FCO$ of compact open subgroups and compare the Mackey cohomological dimension $\cd_{\MCOO} G$ with  the rational discrete cohomological dimension $\cd_{\Q} G$~\cite{cw16} and  that of Bredon functors $\cd_{\O_\FC} G$~\cite{lueckbook}. In particular, we prove:
\begin{thmi} Let $G$ be a t.d.l.c.~group. Then
$$\cd_{\Q} G  \leq \cd_{\MCOO} G \leq \cd_{\O_\FCO} G.$$
\end{thmi}
The first difficulty here lies in the fact that one needs a suitable small category that has pull-backs to be able to define Mackey functors for t.d.l.c.~groups. By analogy with the discrete case, where an arbitrary $G$-set can be regarded as a coproduct of transitive $G$-sets, we begin by considering the category ${}_G\Sigma_\FC,$ which is the free co-product completion of the category of homogeneous $G$-spaces with stabilisers in $\FC$.  As above, $\FC$ is a family of closed subgroups. However, this category is much too large and we need to restrict ourselves to a certain subcategory where the objects are $G$-spaces with finitely many orbits (see Definitions~\ref{def:sco} and~\ref{def:Mackey cat-lindner}). To compare with the rational discrete cohomology of~\cite{cw16}, we also need to restrict ourselves to the family $\FCO$ of compact open subgroups as stabilisers for our objects. The main difficulty here lies in the fact that, for $G$ non-discrete, the trivial subgroup is not in $\FCO$. Note that the tools used in~\cite{mp-n} heavily rely on the fact that the trivial subgroup belongs to the family $\FCO$ of all subgroups of a discrete group $G.$

Related to these various cohomological dimensions is the geometric dimension of a group. In Section~\ref{sec:cw}, we consider classifying spaces for proper actions, and extend some known results about the geometric dimension of a t.d.l.c.~group, see \cite[Theorems 0.1,~0.2]{lueck-meintrup}.

\section{Preliminaries}\label{sec:prelim} 
Throughout, all groups are topological Hausdorff and totally disconnected locally compact (t.d.l.c), unless otherwise stated, and all subgroups are closed. Abstract groups are endowed with the discrete topology.
A {\em family} of subgroups of a given group $G$ is a collection of (closed) subgroups of $G$ that is closed under conjugation and finite intersections, and we generally denote such family by $\FC$, omitting $G$ if the group is unambiguous. 
Similarly, we write $\FCO$ for the family of all compact open subgroups and $\FK$ for that of all compact subgroups. 
\subsection{Mackey systems}  We follow the set-up of~\cite{BB} specialised to t.d.l.c.~groups.
\begin{definition}\label{def:Mackey sys} 
Let $G$ be a t.d.l.c.~group and let $\FC$ be a family of subgroups of $G$. Denoting $\FC(H)=\{U\leq H \,|\, U \in \FC\}$ for $H \in \FC$, a {\em Mackey system} $(\FC, \FO)$ for $G$ consists of the family $\FC$ and a family  $\FO(H) \subseteq \FC(H)$ for every $H \in \FC$, subject to the following axioms:
\begin{itemize}
\item[(i)] $[H:U] < \infty,$ for all $U\in\FO(H),$ 
\item[(ii)] $\FO(U) \subseteq \FO(H),$ for all $U\in\FO(H),$
\item[(iii)] $\FO(H^g)= \FO(H)^g$ for all $g \in G,$
\item[(iv)] $U \cap V \in \FO(V)$ for all $U,V \in \FO(H),$
\item[(v)] $H\in\FO(H)$ for all $H\in\FC$.
\end{itemize}
\end{definition}

If $\FO(H)=\FC(H)$ for all $H\in \FC,$ we denote the Mackey system $(\FC, \FO)$ by $(\FC).$
For example, the family $\FF$ of all finite subgroups of $G$ gives a Mackey system $(\FF)$ for $G$, and it has been studied extensively for infinite discrete groups~\cite{mp-n, degrijse, stjg}.
More generally, for any t.d.l.c.~group $G$, the family $\FCO$ of all compact open subgroups of $G$ gives a Mackey system $(\FCO)$ for $G$. 
\subsection{Categories of $G$-spaces}\label{sec:g-spaces}
For the purposes of this paper, a $G$-{\em space} consists of a topological Hausdorff space $X$ and a continuous left action $\rho\colon G\times X\to X$. We denote the $G$-space $(X, \rho)$ just by its underlying topological space $X$ and write $\rho(g,x)=gx$.  The $G$-spaces form a category with morphisms the continuous $G$-maps between $G$-spaces. For $x\in X$, the closed subgroup $G_x=\{g\in G\mid gx=x\}$ is called {\em stabiliser} of $x$ in $G$ (cf.~\cite[Proposition~3.5]{trans}). The set $R = \{(x,gx)\mid x\in X, g\in G\}$ is an equivalence relation on $X$, whose  equivalence classes are the {\em $G$-orbits} on $X$. In particular, we call an equivalence class $Gx$ the {\em orbit} through $x\in X$. We let $G\backslash X$ denote the set of all $G$-orbits on $X$ and $[G\backslash X]$ a set of orbit representatives. 
Clearly, the set of \mbox{$G$-orbits} $G\backslash X$ is a topological space, called the {\em orbit space} (of the $G$-space $X$), and it is endowed with the quotient topology induced by the map $X \to G\backslash X$ sending an element to its orbit. 
For example, given a closed subgroup $H$ of $G$, the orbit space of the right multiplication action $G\times H\to G$ of $H$ on $G$ is the space of $H$-cosets $G/H$. Moreover, the group $G$ acts continuously on the Hausdorff space $G/H$ (cf.~\cite[Example~1.4 and Proposition~3.3.]{trans}). Any $G$-space $G/H$ with this action is called {\em homogeneous $G$-space}. Recall that there exists a $G$-map $\phi\colon G/H\to G/K$ if and only if $H$ is conjugate to a subgroup of $K$. More precisely, for every $a\in G$ such that $a^{-1}Ha\leq K$, one has a $G$-map $\phi$ sending $g\in G$ to $\phi(gH)=gaK$. Consequently, every $G$-map $G/H\to G/K$ between homogeneous $G$-spaces is continuous.
Note that $G/H$ and $G/K$ are $G$-homeomorphic if and only if $H$ and $K$ are conjugate. 
The {\em orbit category} $\O G$
is the small category whose
objects are the homogeneous $G$-spaces $G/H$, where 
$H$ is a closed subgroup of $G$, and morphisms are $G$-maps among them.
We denote by $\OCG$ the full subcategory of $\O G$ with objects the {homogenous $G$-spaces} $G/H$ with $H\in\FC.$ 

\begin{definition}
Let  ${}_G\Sigma_\FC$ be the free coproduct completion  of $\OCG$.  This is the category of all collections $ (G/H_i)_{i\in I}$
of objects of $\OCG$, where a morphism $$f\colon (G/H_i)_{i\in I} \to (G/K_j )_{j\in J}$$ consists of a map $\hat f : I \to J$, where $I$ and $J$ are indexing sets, possibly empty, and
 a collection of $G$-maps $f_i : G/H_i \to G/K_{\hat f(i)}$ for $i\in I$. Note that $f_i$ is continuous.
 \end{definition}
The empty collection is the initial object of ${}_G\Sigma_\FC$. 
 \begin{rem}[{A description of ${}_G\Sigma_\FC$}]
 Let  ${}_G\mathfrak X_\FC$ be the category whose objects are the $G$-spaces $X$ with stabilisers in $\FC$ such that {each orbit $Gx\subseteq X$ is homeomorphic to the homogeneous space $G/G_x$ and $X\cong\bigsqcup_{x\in[G\backslash X]} Gx\cong \bigsqcup_{x\in[G\backslash X]}G/G_x$. For example, for the family $\FO$ of all open subgroups of $G$, ${}_G\mathfrak X_\FO$ is the category of discrete $G$-spaces.} 
 
{The category ${}_G\Sigma_\FC$ is isomorphic to ${}_G\mathfrak X_\FC.$ On objects, this isomorphism ${}_G\Sigma_\FC \to {}_G\mathfrak X_\FC$ is given by  $(G/H_i)_{i\in I} \mapsto \bigsqcup_{i\in I} G/H_i$ and the inverse  by $X \mapsto (G/G_x)_{x\in[G\backslash X]}$. Now consider a morphism $f\colon (G/H_i)_{i\in I} \to (G/K_j )_{j\in J}$  in ${}_G\Sigma_\FC$ and define
$$\phi\colon \bigsqcup_{i\in I} G/H_i\to \bigsqcup_{j\in J} G/K_j $$
by $\phi(gH_i)=f_i(gH_i)\in G/K_{\hat f(i)}$. On the other hand, let $\phi\colon X\to Y$  be a $G$-map in ${}_G\mathfrak X_\FC$ and choose the set $[G\backslash Y]$ to contain the set $\{\phi(x)\mid x\in[G\backslash X]\}$. One then  maps $\phi$ to the following morphism
$$f\colon  (G/G_x)_{x\in[G\backslash X]}\to  (G/G_y)_{y\in[G\backslash Y]}$$ consisting of
\begin{itemize}
\item $\hat f\colon [G\backslash X]\to[G\backslash Y]$ mapping each $x\in[G\backslash X]$ to $\phi(x)$, and
\item $f_x\colon G/G_x\to G/G_{\phi(x)}$, for every $x\in[G\backslash X]$, mapping $gG_x$ to $gG_{\phi(x)}$.
\end{itemize}}

For $\FK$ the family of all compact subgroups of $G$, the category ${}_G\Sigma_{\FK}$ consists of all proper transitive $G$-spaces and their coproducts (cf.~\cite[Proposition~I.3.19(ii)-(iii)]{trans}). 
 \end{rem} 
 
 \begin{rem}[{${}_G\Sigma_\FC$ as a category of $G$-sets}]\label{fact:spaces} As recalled above, every $G$-map between homogeneous $G$-spaces is continuous, and topological considerations become superfluous, in the sense that the orbit category $\OCG$ is isomorphic to the category whose objects are the transitive $G$-sets with stabilisers in $\FC$ and morphisms the $G$-maps among them. Consequently, ${}_G\Sigma_\FC$ is isomorphic to the category ${}_G\mathrm{Set}_\FC$ (whose objects are the $G$-sets with stabilisers in $\FC$) because they are obtained as free coproduct completions of isomorphic categories.
 \end{rem}

Since ${}_G\Sigma_\FC$ is often too large as category to work with, we define some subcategories that are the analogue to some of the categories that have been considered elsewhere, as recalled in the introduction. For our purposes, the categories $\SCO$ and $\SCO^f$ will be most useful. For completeness, we will list the four obvious ones.

\begin{definition}[{Some subcategories of ${}_G\Sigma_\FC$}]\label{def:sco}
Let $(\FC,\FO)$ be a Mackey system for the t.d.l.c.~group $G$. We define the following categories.
\begin{itemize}[leftmargin=10pt]
\item $\SCO$ is the subcategory of ${}_G\Sigma_\FC$
 with the same objects and sets of morphisms
 $\Mor_{\SCO}(X,Y)=\{f\in\Mor_{{}_G\Sigma_\FC}(X,Y)\mid \forall x\in X: G_x\in\FO(G_{f(x)})\}$, see \cite[Definition 2.6]{BB}. Note that we differ slighly from the notation of \cite{BB}.
 \item $\SCO^{af}$ is the full subcategory of $\SCO$ with objects the $G$-spaces $X$ such that $X^U$ has finitely many $N_G(U)/U$-orbits for all $U\in\FC$, see~\cite[\S~2.2]{ds}.
 \item $\SCO^r$ is  the subcategory of $\SCO$ with the same objects as $\SCO$ and, see~\cite[Definition~2.6]{BB}, $$\Mor_{\SCO^r}(X,Y)=\{f\in \Mor_{\SCO}(X,Y)\mid \text{$f^{-1}(y)$ has finitely many $G_y$-orbits}\}.$$
 \item $\SCO^f$ is the full subcategory of $\SCO$ with objects the $G$-spaces with finitely many orbits, see \cite{BB, dress,mp-n}.
 \end{itemize}
\end{definition}

The superscripts are indicative of the context, i.e. restricted, almost finite, finite, in which these categories of $G$-spaces have been used in the study of Mackey functors.

Note that $\SCO$ is generally not small, but it has small hom-sets. Hence, Yoneda's lemma applies~\cite[\S~III.2]{maclane}.
By Remark~\ref{fact:spaces}, $\SCO$ can be regarded as the category with objects the $G$-sets with stabilisers in $\FC$, and with morphisms the $G$-maps $f\colon X\to Y$ such that $G_x\in\FO(G_{f(x)})$ for every $x\in X$.  This category has products, coproducts and pull-backs, which have the same construction as in ${}_G\mathrm{Set}_\FC,$. On the other hand $\SCO^f$ is a small category because its objects are coproducts of finitely many homogeneous $G$-spaces with stabilisers in $\FC$. Note that $\SCO^f$ also admits  pull-backs. 
Indeed, for any pair of morphisms of homogeneous $G$-spaces $g:G/L\to G/K$ and $g':G/S\to G/K$ in $\SCO$, where $L^g,S^{g'}\leq K$, the pull-back construction yields
$$\xymatrix{\bigsqcup_{x\in L^g\backslash K/S^{g'}} G/L^g\cap S^{g'x^{-1}}\ar[r]\ar[d]&G/S\ar[d]^{g'}\\G/L\ar[r]_g&G/K}$$
Note that $|L^g\backslash K/S^{g'}|$ is finite, by Definition~\ref{def:Mackey sys}(i). 

\begin{rem}[{The category ${}_G\mathcal S_\FCO$}]
Given the Mackey system $(\FCO)$ defined by the family of compact open subgroups of $G$, observe that the categories  ${}_G\mathcal{S}_{(\FCO)}$ and ${}_G\Sigma_\FCO$ coincide. Indeed, given a $G$-map $f\colon X\to Y$ and $x\in X$, then $G_x= G_x\cap G_{f(x)}$ is open in  $G_{f(x)}$ since $G_x$ is open in $G$. We denote these categories simply by ${}_G\mathcal S_\FCO$. 
\end{rem}
\begin{example}
Let us give some examples which show that the above categories of $G$-spaces are genuinely distinct.
Let $G$ be an infinite profinite group. Let $\FC$ be the family of all closed subgroups of $G$ and, for all $H\in\FC$, let $\FO(H)$ be the family of all open subgroups of $H$.
Then, we have strict inclusions
$$\SCOO^f\subset\SCOO^r\subset\SCOO^{af}\subset\SCOO\subset\SCO.$$
The inclusions of the above categories are immediate from the definitions. In particular, $\SCOO^f$ is the category of finite $G$-sets (since $G$ is compact). 
Writing $\Ob(\CS)$ and $\Mor(\CS)$ for the objects and the morphisms of a given category $\CS$, we have:
\begin{itemize}[leftmargin=10pt]
\item $G/1\in\Ob(\SCO)$, but $G/1\notin\Ob(\SCOO)$. 
\item $X=\dst\bigsqcup_{n\in\N}G/G\in\Ob(\SCOO)$ but $\dst\bigsqcup_{n\in\N}G/G\notin\Ob(\SCOO^{af})$, since $|X^U|$ is infinite (countable) for any open subgroup $U$ of $G$ and the Weyl group $N_G(U)/U$ is finite (since $U$ has finite index in $G$).
\item The morphism $\pi:\dst\bigsqcup_{N\trianglelefteq_oG}G/N\to G/G\in\Mor(\SCOO^{af})$ but $\pi\notin\Mor(\SCOO^r)$, since there is a unique fibre and it does not have finitely many $G$-orbits.
\item The object $X=\dst\bigsqcup_{N\trianglelefteq_oG}G/N$ is in $\Ob(\SCOO^r)$ but not in $\Ob(\SCOO^f)$. \end{itemize}
\end{example}


\subsection{Modules over a category}\label{cat}
We now collect a few facts about modules over a category in a general set-up, following the approach by St.~John-Green~\cite[\S~2.1]{stjg}.
Throughout this section, we let $R$ be a commutative ring with unit.
\begin{definition}
Let $\D$ be a small abelian category. We say that {\em $\D$ satisfies condition~$(A)$} if, for any two objects $x,y \in \D$, the set of morphisms, denoted $[x,y]_{\D}$, is a free abelian group.
\end{definition}
For any small abelian category~$\D$ with~$(A)$ we define the category of  $\D$-modules over~$R$, denoted~$\DMod_R$ as the category of contravariant functors 
$$\D \to \RMod,$$
where $\RMod$ denotes the category of left $R$-modules. The morphisms in $\DMod_R$ are the natural transformations.
Since $\RMod$ is an abelian category and~$\D$ is a small abelian category, the category of contravariant $\D$-modules inherits many properties from~$\RMod$. Most notably, we have that limits and colimits exist and that products and coproducts of exact sequences are exact. Furthermore,  a sequence
$$A(-) \to B(-) \to C(-)$$ 
of $\D$-modules is exact at $B(-)$ if and only if, for every $x \in \D$, the  sequence
$$A(x) \to  B(x) \to C(x)$$
 of $R$-modules is exact at $B(x).$  For each $y \in \D$ one defines the contravariant functor $R[-,y]_{\D}$ as follows:
$$R[-,y]_{\D}(x)= R \otimes_{\Z} [x,y]_{\D}.$$
These $\D$-modules are the building blocks for the free $\D$-modules over $R$ and are defined as adjoints to some forgetful functor. See~\cite{stjg} for the details.
If $R$ is unambiguous from the context, we simply write $\D$-modules instead of $\D$-modules over $R$. Hence free $\D$-modules are direct sums of modules of the form $R[-,y]_{\D},$ and projective modules are direct summands of free $\D$-modules. The following result is an adaptation of Yoneda's Lemma and a standard result in the case of the orbit category, see for example \cite[p.~9]{mislin}.

\begin{prop}\label{prop:enough}\cite[Lemma 2.2]{stjg} Let $M(-) \in \Ob(\DMod)$ and $x \in \Ob(\D)$. Then there is a natural isomorphism
$$\begin{array}{ccc}
\Hom_{\DMod}(R[-,x]_{\D},M(-)) & \cong & M(x) \\
\eta& \mapsto & \eta_x(id_x).
\end{array}$$
In particular, the category of contravariant $\D$-modules has enough projectives.
\end{prop}

\begin{rem}\label{rem:cd}
Proposition \ref{prop:enough} implies that we can do cohomology in the standard manner, see for example \cite[\S 4.1]{weibel}. In particular, let $M(-)$ and $N(-)$ be $\D$-modules and let
$$\cdots  \to P_{n}(-)  \to P_{n-1}(-) \to P_{0}(-)\to M(-) \to 0$$
be a projective resolution, i.e. an exact sequence of projective $\D$-modules $P_i(-)$,  of $M(-).$ We define the Ext-groups to be 
$$\Ext^k_{\D}(M(-),N(-)) := \Hn^k \Hom_{\DMod}(P_{\ast}(-), N(-)),$$
and the projective dimension of a $\D$-module to be
$$\pd_{\D} M(-) = \sup\{k \,|\, \Ext^k_{\D}(M(-), ?) \neq 0\}.$$
This is equivalent to $\pd_{\D} M(-)$ being the shortest length of a projective resolution of $M(-).$  In particular, if $\pd_{\D} M(-) = n,$ there  is a projective resolution
$$ 0 \to P_{n}(-)  \to P_{n-1}(-) \to P_{0}(-)\to M(-) \to 0$$
and in any other projective resolution $Q_{\ast}(-) \epi M(-),$ the kernel $Q_{n-1}(-) \to Q_{n-2}(-)$ is projective. Furthermore, if the category of $\D$-modules has free objects, the fact that projectives are direct summands of free modules and an Eilenberg swindle argument imply that the terms $P_{\ast}(-)$ above can be chosen to be free, see for example \cite[pp.187/188]{brown}.
\end{rem}

\begin{example}[\bf Mackey spaces] \label{ex:mackey-A} Let $\Msys$ be a Mackey system. The category of $G$-Mackey spaces $\MCOG$ with objects in $\SCO^{f}$, as in Definition~\ref{def:Mackey cat-lindner} below, satisfies condition~(A) (cf.~Remark~\ref{rem:freeab}).
\end{example}

\begin{example}[\bf Bredon modules] \label{ex:Bredon-A}  Analogously to the discrete case, see~\cite[\S~II.9.1]{trans} and~\cite[\S~II.9]{lueckbook},  one can define Bredon modules for t.d.l.c.~groups in a straightforward manner as modules over the orbit category $\O_\FC G$. The category of {\em Bredon modules} $\OCG\,\mbox{-}\Mod$ is the category whose objects are the contravariant functors
$M(-): \OCG \to \Ab$, and the morphisms are natural transformations. 
This is an abelian category with enough projectives. The projective building blocks of any free module are the (contravariant) functors
$$P_K(-) = \Z [- , G/K]_{\OCG},$$
where $K \in \FC$. That is, for all $H \in \FC$, the evaluation $P_K(G/H) = \Z [G/H,G/K]_{\OCG}$ is the free abelian group on the set $[G/H,G/K]_{\OCG}.$ Note that
$$P_K(G/H) \cong \Z((G/K)^H)$$
is an abstract permutation $N_G(H)/H$-module. In particular, $P_H(G/H)\cong\Z N_G(H)/H$ and $P_K(G/H)$ is zero if~$K$ does not contain any conjugate of~$H$.

Following~\cite[Remark~2.0.7]{stjg} one can show that the $\OCG$-module category is isomorphic to the~$\OCGp$-module category, where~$\OCGp$ has the homogeneous $G$-spaces with stabilisers in~$\FC$ as objects and the morphism set between two such $G$-spaces is the free abelian group generated by the $G$-maps between them. This category satisfies condition~$(A)$.
\end{example}


\section{Mackey functors for t.d.l.c.~groups}\label{sec:fcg} 

We will define a Mackey functor using the categorical approach used by Lindner, see, for example~\cite{mp-n,tw-structure} as this is the most useful viewpoint for the present purpose. 
Our results in the sequel require a mild extra assumption on the Mackey systems considered, namely that the family $\FC$ be large enough in the following sense. 

\begin{definition}\label{def:Mackey cat-lindner} Let $(\FC,\FO)$ be a Mackey system of $G$ such that $\FCO\subseteq\FC$. We define $\MCOG$ to be the category of $G$-Mackey spaces as follows:
\begin{itemize}
\item The objects are the objects of $\SCO^{f}$.
\item The morphism sets between two objects $X,Y$, denoted $[X,Y]_{\MCOG}$, is the Grothendieck group of the commutative (additive) monoid consisting of all equivalence classes of diagrams of the form
$$\xymatrix{X&Z\ar[l]_\alpha\ar[r]^\beta&Y},$$
where $\alpha$ and $\beta$ are morphisms in $\SCO^f$.  
\end{itemize}
\end{definition}
Recall that two diagrams are equivalent, say $$\xymatrix{X&Z\ar[l]_\alpha\ar[r]^\beta&Y}\sim\xymatrix{X&Z'\ar[l]_{\alpha'}\ar[r]^{\beta'}&Y},$$
if there exists {an isomorphism} $\gamma\in\Mor_{\SCO^f}(Z,Z')$ such that the following diagram commutes:
$$\xymatrix{&Z\ar[dl]_\alpha\ar[dd]^\gamma\ar[dr]^\beta\\X&&Y\\&Z'\ar[ul]_{\alpha'}\ar[ur]^{\beta'}}.$$

The addition of two morphisms from $X$ to $Y$ is given by their disjoint union, i.e.,
$$[\xymatrix{X&Z\ar[l]_\alpha\ar[r]^\beta&Y}]+[\xymatrix{X&Z'\ar[l]_{\alpha'}\ar[r]^{\beta'}&Y}]=[\xymatrix{X&(Z\sqcup Z')\ar[l]_<<<{\alpha\sqcup\alpha'}\ar[r]^<<<{\beta\sqcup\beta'}&Y}]$$
and the zero element is
$[\xymatrix{X&\emptyset\ar[l]\ar[r]&Y}]$, where the arrows denote the unique maps from the initial object $\emptyset$ to $X$ and $Y$, respectively. 

Morphisms in $\MCOG$ can be composed (right-to-left) using pull-backs:
$$[\xymatrix{Y&V\ar[l]_\gamma\ar[r]^\delta&Z}]\circ[\xymatrix{X&U\ar[l]_\alpha\ar[r]^\beta&Y}]=[\xymatrix{X&W\ar[l]_{\alpha\tilde{\gamma}}\ar[r]^{\delta\tilde{\beta}}&Z}]$$
where $W,\tilde{\gamma}$ and $\tilde{\beta}$ are given by the pull-back in $\SCO^f$ 
$$\xymatrix{&&W\ar@{.>}[dl]_{\tilde{\gamma}}\ar@{.>}[dr]^{\tilde{\beta}}\\&U\ar[dl]_\alpha\ar[dr]^\beta&&V\ar[dl]_\gamma\ar[dr]^\delta\\X&&Y&&Z}.$$
\begin{rem}[{The basic morphisms}]\label{rem:freeab}
Since $G$-Mackey spaces have finitely many $G$-orbits by definition, every group $[X,Y]_{\MCOG}$ is free abelian with basis consisting of morphisms, often called {\em basic}, of the form 
$$[\xymatrix{G/H&G/L\ar[l]_g\ar[r]^h&G/K}],$$ where $L\mapsto gH$ and $L\mapsto hK$ for some $g,h\in G$ such that $L^g\leq H$ and $L^h\leq K$ (see \cite[Proposition~2.2]{tw-structure}). Note that any basic morphism $[\xymatrix{G/H&G/L\ar[l]_g\ar[r]^h&G/K}]$  is equivalent to  $[\xymatrix{G/H&G/L^g\ar[l]_1\ar[r]^{g^{-1}h}&G/K}].$ 
The following observation is useful for later computations: if $\{1\}\in\FC$ then $[G/1,G/K]_{\MCOG}\cong_{\mathcal Ab} \Z[G/K]$ given by $$[\xymatrix{G/1&G/1\ar[l]_1\ar[r]^g&G/K}]\mapsto gK.$$
\end{rem}
\begin{definition}\label{def:macmod}
A contravariant functor from $\MCOG$ to ${\mathcal Ab}$ that preserves finite products\footnote{Recall that a functor between additive categories  is an additive functor if and only if it preserves  finite products (cf.~\cite[Corollary~2.5(2)]{lax}.} is said to be a {\em Mackey functor} for $(\FC,\FO).$
With a slight abuse of notation, we define $\MCOG\,\mbox{-}\Mod$ to be the category of Mackey functors for $(\FC,\FO).$
\end{definition}
The preadditive category of modules over~$\MCOG$   is small and satisfies condition~$(A)$ (see Remark~\ref{rem:freeab} and~Section~\ref{cat}). 
It turns out that every finitely generated free module over $\MCOG$ is a coproduct of contravariant functors of the form $[-,X]_{\MCOG}$ for objects $X$ in $\MCOG$.
Consequently, the set $$\{[-,X]_{\MCOG}\mid X\in\MCOG\}$$ is a family of finitely generated projective generators in $\MCOG\,\mbox{-}\Mod$ (cf.~Section~\ref{cat}).
\begin{rem}[{Equivalent approaches to Mackey functors}] 
 One can of course define Mackey functors for t.d.l.c.~groups using the approaches due to Dress~\cite{dress} and Green~\cite{green}, see also~\cite{BB, nakaoka, tw-structure}. 
 In many cases, the three definitions have been shown to be equivalent. It is a routine exercise to observe that this equivalence also holds in the present context.
\end{rem}

\begin{rem}[{The module structure on $M(G/H)$}]\label{rem:mod_str}
Let $M\colon\MCOG\to\mathcal Ab$ be a Mackey functor and $H\in\FC$. One has the group homomorphism
$$N_G(H)/H\to\mathrm{Aut}(M(G/H)),\quad gH\mapsto M([G/H\stackrel{1}{\leftarrow} G/H\stackrel{g}{\rightarrow} G/H])$$
defining a left $N_G(H)/H$-module structure on $M(G/H)$ . 

Let us examine the special case $M=[-,G/K]_{\MCOG}$ with $\{1\}\in \FC$. Hence we are interested in the left $G$-action on the free abelian group $[G/1,G/K]_{\MCOG}$, which is freely generated by basic morphisms of the form $[\xymatrix{G/1&G/1\ar[l]_1\ar[r]^g&G/K}]$.  For every~$x\in G$ one has 
\begin{align*}
    x\cdot[\xymatrix{G/1&G/1\ar[l]_1\ar[r]^g&G/K}]&=\xymatrix@C=1em@R=1em{&&G/1\ar@{.>}[dl]_{x^{-1}}\ar@{.>}[dr]^{1}\\&G/1\ar[dl]_1\ar[dr]^x&&G/1\ar[dl]_1\ar[dr]^g\\G/1&&G/1&&G/K}\\
    &= [\xymatrix{G/1&G/1\ar[l]_{x^{-1}}\ar[r]^g&G/K}]\\
    &= [\xymatrix{G/1&G/1\ar[l]_{1}\ar[r]^{xg}&G/K}]
\end{align*}
By Remark~\ref{rem:freeab}, the left $\Z[G]$-module $[G/1,G/K]_{\MCOG}$ is isomorphic to the permutation module $\Z[G/K].$ 
\end{rem}
\subsection{The Burnside Mackey functor}
From \cite[\S~8]{tw-structure} the Burnside Mackey functor is important when studying the structure of Mackey functors. In our context, we distinguish two situations, depending on whether or not $G\in\FC$.
If $G\in\FC$, then $G/G=\bullet$ is the terminal object in $\SCOF$, and we can consider the Mackey functor $[-,\bullet]_{\MCOG}: \MCOG \to {\Ab}$ in $\MCOG\,\mbox{-}\Mod$.

If $G \not\in \FC$, then $\SCOF$ has no terminal object. 
In this case, we add a terminal object $\bullet$ to obtain a larger category, say $\SCOO^{f,\bullet}$, and we add morphisms of the form $X\stackrel{\alpha}{\leftarrow} Z \to\bullet$ with $\alpha$ in $\SCOF$ to define free abelian groups $[X, \bullet]_{\SCOO^{f,\bullet}}$, and hence a Mackey functor $[-,\bullet]_{\MCOG}: \MCOG \to {\Ab}$ (with the abuse of notation of using $\MCOG$ for this larger category as well).

In both cases, we write $[-,\bullet]_{\MCOG} = \underbar B^G_{(\FC,\FO)}(-)$ and refer to it as the  {\em Burnside Mackey functor} associated to the Mackey system $(\FC,\FO)$. 
If there is no confusion about the t.d.l.c.~group and the Mackey system, we write $\underbar B(-)$ instead of $\underbar B^G_{(\FC,\FO)}(-)$.
\begin{prop}\label{prop:bproj} If~$G\in\FC$ then the Burnside Mackey functor $\underbar{B}(-)$ is projective in~$\MCOG\,\mbox{-}\Mod$.
\end{prop}
\begin{proof} This follows from  Yoneda's Lemma and the fact that $[-,\bullet]_{\MCOG}$ is free, see Subsection~\ref{cat}.
\end{proof}
\begin{rem}[{A description of $\underbar B(G/H)$}]\label{r:burnside}
For every $H\in\FC$, we have
$$\underbar B(G/H)\cong \bigoplus_{L\in \FO(H)\atop \text{up to $H$-conjugacy}}\Z_L .$$
Indeed, $\underbar B(G/H)$ is the free abelian group with basis given by the basic morphisms $[\xymatrix{G/H&G/L\ar[r]\ar[l]&\bullet}]$ with $L\in \FO(H)$ (up to $H$-conjugacy).
\end{rem}
\section{Cohomological dimensions}
In this section we investigate and compare various cohomological dimensions of the functors defined above.
We start with reviewing some basic facts of the rational discrete cohomology of Castellano and Weigel~\cite{cw16}.
We then proceed to prove our main results and raise a few open questions.

\subsection{Rational discrete cohomological dimension}
Let $G$ be a t.d.l.c.~group and let $\Q[G]\,\mbox{-}\operatorname{Dis}$ denote the category of {\em discrete $\Q[G]$-modules}~\cite{cw16}. Namely, $\Q[G]\,\mbox{-}\operatorname{Dis}$ is the full subcategory of $\Q[G]\,\mbox{-}\Mod$ whose objects are the (left) $\mathbb Q[G]$-modules $M$ such that $G\times M\to M$ is continuous if $M$ carries the discrete topology. For example, given  an open subgroup $U\leq G$, the permutation $\Q[G]$-module $\Q[G/U]$ belongs to $\mathbb Q[G]\,\mbox{-}\operatorname{Dis}$. By \cite[Proposition~3.2]{cw16}, $\Q[G]\,\mbox{-}\operatorname{Dis}$ is an abelian category with enough projectives and injectives.  
In particular, a discrete \mbox{$\mathbb Q[G]$-module} $P$ is projective if and only if, for some index set $I$, $P$ is a direct summand of a discrete permutation $\mathbb Q[G]$-module of the form
\begin{equation}\label{eq:perm}
\bigoplus_{i\in I}\mathbb Q[G/U_i],
\end{equation}
where $U_i$ is a compact open subgroup of $G$ for all $i\in I$. Every $\mathbb Q[G]$-module of the form \eqref{eq:perm} is called {\em proper discrete permutation $\mathbb Q[G]$-module}.
The notion of projective dimension in $\mathbb Q[G]\,\mbox{-}\operatorname{Dis}$ is defined in the usual way, see Remark \ref{rem:cd},
and the {\em rational discrete cohomological dimension} of $G$ 
\begin{equation}\label{eqn:cd-dis}\cd_{\mathbb Q}G:=\mathrm{pd}_{\mathbb Q[G]\,\mbox{-}\operatorname{Dis}}(\mathbb Q).\end{equation}
is the projective dimension of the trivial discrete $\mathbb Q[G]$-module $\Q$.
From the definition, see {\cite[Proposition~3.7(a)]{cw16}}, we record two immediate consequences.
\begin{rem}\label{rem:rdcd} Let $G$ be a t.d.l.c.~group. The following hold:
\begin{enumerate}
\item[(a)] $\cd_{\mathbb Q}G = 0$ if and only if $G$ is compact.
\item[(b)] If $G$ is discrete, then $\cd_{\mathbb Q}G$ is the classical rational cohomological dimension.
\end{enumerate}
\end{rem}
\subsection{Cohomological dimension of the Burnside Mackey functor}
Throughout this section, we let $G$ be a t.d.l.c.~group, and we consider Mackey systems~$\Msys$ with~$\FC$ a family of compact subgroups containing~$\FCO.$ 
By Remark~\ref{fact:spaces}, we can consider the objects in~$\SCO^{f}$ to be abstract $G$-sets instead of $G$-spaces, which simplifies the comparison with Bredon cohomology.

Since by definition $\MCOG$ is a small abelian category with $(A),$ we can apply the results and facts of Section~\ref{cat} to $\MCOG\,\mbox{-}\Mod$. For an arbitrary module $M(-) \in \MCOG\,\mbox{-}\Mod$ we define
\begin{equation}
\Hn^*_{\MCO}(G; M(-)) := \Ext^*_{\MCOG}(\underbar B(-), M(-)).
\end{equation}
Hence, we define the {\em Mackey cohomological dimension} of $G$ to be the projective dimension of the Burnside Mackey functor
\begin{equation}
\cd_{\MCO} G := \pd_{\MCOG} \underbar{B}(-).
\end{equation}
\begin{prop}\label{prop:cdM} If $G$ is compact then $\cd_{\MCO}G = 0$.
\end{prop}
\begin{proof} This follows from the fact that $\cd_{\MCO}G = 0 $ if $G/G \in\Ob(\SCO^{f})$, see Proposition~\ref{prop:bproj}.
\end{proof}

Let now consider the Mackey system $(\FCO,\FCO)$ and denote $\mathcal M_{(\FCO,\FCO)}G\,\mbox{-}\Mod$ simply by $\mathcal M_{\FCO}G\,\mbox{-}\Mod$. We will now show that we can evaluate projective resolutions of the Burnside Mackey functor in $\mathcal M_{\FCO}G\,\mbox{-}\Mod$ at every compact subgroup $K$ of $G.$ 
This is an adaptation of a similar construction for Bredon functors communicated to us by Ged Corob Cook~\cite{cc}.
First note that by~\cite[Lemma 3.5]{luecksurvey} it follows that every compact subgroup $K$ of $G$ is contained in a compact open subgroup of~$G$. Now apply \cite[2.1.4 (d)]{RZ} to conclude that 
$$K = \bigcap_{K\leq U \in \FCO} U.$$

Let $U$ and $V$ in $\FCO$ containing $K$ and let $G/V \to G/U$ be a morphism in~$\mathcal S_{\FCO}^{f}$. For every~$L\in\FCO$, we obtain a group homomorphism 
$$ [G/U,G/L]_{\MCOOG}\to [G/V,G/L]_{\MCOOG}$$ in the usual way. Explicitly, we map a basic morphism $[\xymatrix{G/U&G/H\ar[r]\ar[l]&G/L}]$ to the morphism $[\xymatrix{G/V&X\ar[r]\ar[l]&G/L}]$ defined by the pull-back
$$\xymatrix{
G/V\ar[d]&X\ar@{-->}[l]\ar@{-->}[d]&\\
G/U&G/H\ar[l]\ar[r]&G/L}.$$
Hence, taking colimits, we define:
$$[G/K,G/L]_{\MCOOG}: = \colim_{K \leq U \in \FCO} [G/U,G/L]_{\MCOOG}.$$ 
The properties of the Burnside functor (see \cite[Section 5]{ds}) allow us to define analogously
$$\underbar{B}(G/K): = \colim_{K\leq U \in \FCO} \underbar{B}(G/U).$$

This now leads us to the following:

\begin{prop}\label{prop-mackey-ev-comp} Let $F_*(-) \to \underbar{B}(-) \to 0$ be a free resolution of the Burnside Mackey functor in $\MCOOG\,\mbox{-}\Mod$ and let $K$ be a compact subgroup of $G.$ 
Then, evaluation at $K$ yields an exact sequence $F_*(G/K) \to \underbar{B}(G/K) \to 0$ of abelian groups.
\end{prop}

\begin{proof} This follows directly from the fact that direct limits commute with direct sums and are exact. 
\end{proof}

\subsection{Bredon cohomology for t.d.l.c.~ groups}
\begin{definition} Let $\FC$ be a family of closed subgroups of a t.d.l.c.~group $G$ and
let $\Z(-)$ denote the constant Bredon module. For $M(-) \in \Ob(\OCG\,\mbox{-}\Mod)$ we define
$$\Hn^*_{\O_{\FC}}(G; M(-)) := \Ext^*_{\O_{\FC}G}(\Z(-), M(-))\quad\text{and}\quad\cd_{\O_{\FC}}G = \pd_{\OCG\,\mbox{-}\Mod} \Z(-).$$
\end{definition}

We now compare $\cd_{\O_\FC} G$ with the rational discrete cohomological dimension $\cd_{\Q}G$ and the Mackey cohomological dimension $\cd_{\MCO}G$ defined above. 

\begin{prop}\label{cd-rational-Bredon}
Let $G$ be a t.d.l.c.~group and let $K$ be a compact subgroup of $G.$  Then 
$$\cd_{\Q} N_G(K)/K \leq \cd_{\OCO } G.$$
In particular,
$$\cd_{\Q} G \leq \cd_{\OCO } G.$$
\end{prop}

\begin{proof} First note that, given $U \in \FCO$, if $U$ does not contain any conjugate of $K$ then $(G/U)^K$ is empty, otherwise the $N_G(K)/K$-set $(G/U)^K$ decomposes as a disjoint union of transitive sets of the form $N_G(K)\big/\big(U^{x^{-1}}\cap N_G(K)\big)$ for some $x \in G$ such that $K^x\leq U$. 
Now, since $G$ is Hausdorff, $N_G(K)$ is closed in $G$ and hence $U^{x^{-1}}\cap N_G(K)$ is compact and open in $N_G(K)$.
It follows that $(U^{x^{-1}}\cap N_G(K))/K$ is a compact open subgroup of $N_G(K)/K$, and so $\Z(G/U)^K$ is a proper discrete permutation $N_G(K)/K$-module.

Suppose  that $\cd_{\OCO} G =n < \infty.$ 
Then we can choose a length $n$ resolution by free $\OCOG$-modules, each a direct sum of modules of the form $\Z[ - ,G/U]_{\OCO G}.$ 
Now, by~\cite{cc} we can evaluate this at every compact subgroup $K$, and this yields an exact sequence of abelian groups. By the above paragraph, this is a length $n$ resolution of $\Z$ by proper discrete permutation $N_G(K)/K$-modules. 
Tensoring by $\Q$ gives a length $n$ resolution of $\Q$ by proper discrete permutation $\Q[N_G(K)/K]$-modules. The assertion follows (cf.~Equation \eqref{eqn:cd-dis}).
\end{proof}

Also note that the result above for $\cd_{\OCO } G$ also follows  from our Main Theorem.

\begin{question}\label{dim-qu1} For which t.d.l.c.~groups do we have $\cd_{\Q}G=\cd_{\OCO}G$?
\end{question}
There are many examples of discrete groups where $\cd_{\Q}G<\cd_{\OCO}G,$ see for example~\cite{leary-nucinkis}. 
We are not aware of such an example for non-discrete t.d.l.c.~groups that is not merely a slight modification of the known ones.
Furthermore, Question~\ref{dim-qu1} also relates to an algebraic version of the Kropholler--Mislin conjecture~\cite{kropholler-mislin}, which translates in our context as follows.
\begin{question}\label{dim-qu2}  Are there t.d.l.c~groups with $\cd_{\Q}G < \infty$ but with $\cd_{\OCO}G = \infty$?
\end{question}

We now compare the Bredon $\FC$-cohomological dimension with the cohomological dimension over the Mackey category. 
As above, let $\FC$ be a family of subgroups of $G$ such that $\FCO \subseteq \FC \subseteq \FK$, $\O_\FC G$ the associated orbit category and $(\FC,\FO)$ a Mackey system. We follow the argument of \cite[Corollary~3.9]{mp-n}, based on the existence of appropriate restriction and induction functors which we now define.

Let $\sigma: \O_\FC G \to \MCOG$ be the functor given by the identity 
on objects and which sends a $G$-map $G/S \stackrel{g}{\rightarrow}G/K$ to the basic morphism
$G/K \stackrel{1}{\leftarrow} G/K \stackrel{g}{\rightarrow}  G/S$. The functor $\sigma$ yields the following restriction and induction functors:

\begin{align}\nonumber
\mathrm{res}_\sigma\colon \MCOG\,\mbox{-}\Mod &\rightarrow \O_\FC G\,\mbox{-}\Mod\\ \nonumber
M&\mapsto M\circ\sigma\\
&\\\nonumber
\mathrm{ind_\sigma}\colon \O_\FC G\,\mbox{-}\Mod &\rightarrow \MCOG\,\mbox{-}\Mod\\\nonumber
T&\mapsto [-,\sigma(?)]_{\MCOG}\otimes_{\FC}T(?)
\end{align}
Here $?$ denotes the placeholder for an element in $\O_\FC$. Explicitly, for every object $X$ in $\SCO^f$, we get an abelian group
$$\mathrm{ind_\sigma}\,T(X):= [X,\sigma(?)]_{\MCOG}\otimes_{\FC}T(?)=\bigoplus_{K\in\FC}  {[X,G/K]_{\MCOG}\otimes_{\Z}T(G/K)}_{\big/\sim}$$
where $a\otimes\alpha^\ast(m)\sim \alpha_\ast(a)\otimes m$ for all morphisms $\alpha\colon G/H\to G/K$ and elements $m\in T(G/K)$ and $a\in[X,G/H]_{\MCOG}$. 

In particular, if $X=G/H$ for some $H\in\FC$, it follows from Remark~\ref{rem:freeab}, that
\begin{align}\label{eq:ind}
\mathrm{ind_\sigma}\,T(G/H)&=[G/H,\sigma(?)]_{\MCOG}\otimes_{\FC}T(?)\\ \nonumber
&=\bigoplus_{L\in \O(H)\atop \text{up to $H$-conjugacy}}\Z\otimes_{N_H(L)/L}\big({\Z[G/L,?]_{\O_\FC G}}\otimes_{\FC}T(?)\big)\\ \nonumber
&=\bigoplus_{L\in \O(H)\atop \text{up to $H$-conjugacy}}\Z\otimes_{N_H(L)/L}T(G/L)
\end{align}
(cf. \cite[Proposition~3.6 and Theorem~3.7 ]{mp-n}).
In particular, in combination with Remark~\ref{r:burnside} this yields:
\begin{equation}\label{eq:indB}\mathrm{ind_\sigma}\Z(G/H) \cong \bigoplus_{L\in \O(H)\atop \text{up to $H$-conjugacy}}\Z_L\,  \cong\, \underbar{B}(G/H)\end{equation}
and 
$$\mathrm{ind_\sigma}\Z(-) \cong \underbar{B}(-).$$
Restriction and induction come together with the usual adjoint isomorphism. In particular, since restriction is exact, induction maps projectives to projectives.

\begin{thm}\label{cd-Mackey-Bredon} Let $G$ be a t.d.l.c.~group and $\FC$ a family of compact subgroups containing $\FCO.$ Then  $$\cd_{\MCO} G \leq \cd_{\O_\FC} G,$$
and for each $M(-) \in \Ob(\MCOG\mbox{-}\Mod)$ there is a natural isomorphism
$$ \Hn^*_{\O_\FC}(G; M(-)) \cong \Hn^*_{\MCO}(G; M(-)).$$
\end{thm}

\begin{proof}

Let $P_{\ast}(-)\to\Z(-)$  be a projective resolution in $\O_\FC G\,\mbox{-}\Mod$ of minimal length. By applying the induction functor, we obtain a chain complex $\mathrm{ind_\sigma}\,P_{\ast}(-)\to B(-)$ in $ \MCOG\,\mbox{-}\Mod $ where each $\mathrm{ind_\sigma}\,P_{\ast}(-)$ is projective.  It suffices to prove that the complex is exact. Using the standard Eilenberg swindle argument, we may assume that each $P_{\ast}(-)$ is free (cf.~Remark~\ref{rem:cd}).

To this end we prove that the functor $\mathrm{ind_\sigma}\,P_{\ast}(-)\to \underbar B(-)$ is exact when evaluated at each proper discrete transitive $G$-space $G/H$. 
By \eqref{eq:ind} this gives
$$\mathrm{ind_\sigma}\,P_{\ast}(G/H) \cong \bigoplus_{L\in \O(H)\atop \text{up to $H$-conjugacy}}\Z\otimes_{N_H(L)/L}P_\ast(G/L) \to \bigoplus_{L\in \O(H)\atop \text{up to $H$-conjugacy}} \Z_L.$$
Hence it suffices to show that each $P_{\ast}(G/L)\to\Z$ is split when restricted to every finite subgroup $N_H(L)/L$ of $N_G(L)/L.$
This follows from the argument of~\cite[Theorem~3.2]{N3}.
\end{proof}

This now also allows us to compare the discrete rational cohomological dimension with the cohomological dimension in the Mackey category. 
\begin{cor}\label{cor-mackey-rational-comp} For every t.d.l.c.~group $G$ one has $\cd_{\MCOO}G \geq \cd_\Q G.$
\end{cor}

\begin{proof}  By Proposition~\ref{prop-mackey-ev-comp} we can evaluate any free resolution $F_*(-) \to \underbar{B}(-) \to 0$ of Mackey functors at any compact subgroup. 
Evaluation at the trivial subgroup~$\{1\}$ yields a resolution $F_*(G/1) \to \underbar{B}(G/1) \to 0.$ 
Furthermore, by Remark~\ref{rem:mod_str}, for every compact open subgroup $L\leq G$ we have $[G/1,G/L]_{\MCOOG}\cong_G \Z[G/L]$  and  $\underbar B(G/1) \cong_G \Z$ (cf. Equation~\eqref{eq:indB}).  Thus  $F_*(G/1) \to \underbar{B}(G/1) \to 0$ is a resolution of $\Z$ by permutation modules with compact open stabilisers.
Hence, tensoring with $\Q$ any finite length resolution of $\underbar{B}(-)$ by free Mackey functors gives rise to a finite length resolution of $\Q$ by rational discrete permutation modules.
\end{proof}

Putting the above results together, we obtain the result asserted in the main theorem, comparing the rational discrete, Mackey and Bredon cohomological dimensions of a t.d.l.c.~group~$G$.
\begin{proof}[Proof of the Main Theorem]
The statement now follows from Theorem~\ref{cd-Mackey-Bredon} applied to the families $\MCOO$ and $\OCO$ and Corollary~\ref{cor-mackey-rational-comp}.
\end{proof}
\begin{cor}
$\cd_{\MCOO}G=0$ if and only if $G$ is compact.
\end{cor}
\begin{proof} By Proposition~\ref{prop:cdM}, it suffices to prove the ``only if" part of the claim. Assume $\cd_{\MCOO}G=0$. By Corollary~\ref{cor-mackey-rational-comp}, $ \cd_\Q G=0$. By Remark~\ref{rem:rdcd}(a), $G$ is compact. 
\end{proof}


\section{Classifying spaces for families}\label{sec:cw}

In this section, we revisit the related notion of {\em geometric dimension} to the cohomological dimensions considered above, and extend some known results to the class of t.d.l.c.~groups. In order to do so, we introduce some additional notation to the one used above.

Classifying spaces for groups relative to various families of subgroups have been studied extensively for discrete and locally compact groups. A good resource for different approaches  is L\"uck's survey on classifying spaces~\cite{luecksurvey}. There are two common approaches, one via numerable spaces, and another as terminal object in the $G$-homotopy category of $G$-CW-complexes with isotropy in a family $\FF.$ 
It follows from \cite[Lemma 3.5]{luecksurvey} that for a t.d.l.c.~group $G$ and a family of closed subgroups $\FC$ of $G$ such that $\FCO \subseteq \FC \subseteq \FK$, these two models agree up to $G$-homotopy equivalence and, furthermore, there is a $G$-homotopy equivalence $E_{\FC} G \simeq E_{\FCO} G.$ 

Hence, referring to \cite[\S1.1]{luecksurvey} for the background on \mbox{CW-complexes} for groups, we define the following:
\begin{definition} \cite{cw16} Let $G$ be a t.d.l.c.~group and let $\FC$ be a family of subgroups such that $\FCO \subseteq \FC \subseteq \FK$.
Given a $G$-CW-complex $X$ we say that $X$ is a model for $E_{\FC} G$ if the following three conditions hold:
\begin{itemize}
\item[(i)] $X$ is contractible
\item[(ii)] all the isotropy groups of $X$ belong to $\FC$ and
\item[(iii)] for all $H \leq G$, we have
$$ X^H  \begin{cases} \mbox{is contractible,}& \mbox{ if } H \in \FC, \\
			=\emptyset, &\mbox{ otherwise.} \end{cases}$$
\end{itemize}
The {\em geometric dimension} of $G$ relative to the family $\FC$, denoted $\gd_{\FC} G$, is the smallest dimension of a model for $E_{\FC} G.$ 
If there is no finite-dimensional model for $E_{\FC} G$ then we set $\gd_{\FC} G = \infty.$
\end{definition}

\begin{example}
(a) Tree almost automorphism groups $\mathcal A^D_{qr}$ have infinite geometric dimension \cite[Remark~1.7]{st}.

(b) Hyperbolic t.d.l.c.~groups have finite geometric dimension~\cite{ccc}.

(c) By Bass--Serre theory, $\gd_{\FCO} G=1$ if and only if $G$ is isomorphic to the fundamental group of a non-trivial finite graph of profinite groups. 

(d) Let $p$ be a prime. For every $n\geq 2$ the Bruhat--Tits building associated to $\mathrm{SL}_n(\Q_p)$ is a model for proper actions of smallest dimension $n$.
\end{example}

The following result is well-known for discrete groups and the proof translates to t.d.l.c.~groups verbatim, see \cite[pp.~10ff]{misval}.
\begin{prop}\label{easy-prop} 
Let $G$ be a  t.d.l.c.~group. Then $\cd_{\O_\FC}G \leq \gd_{\FC}G.$ \qed
\end{prop} 
We also have a partial converse for this proposition:
\begin{prop}\label{cd-gd-prop}
Let $G$ be a t.d.l.c.~group and $\FC$ a family of subgroups such that $\FCO \subseteq \FC \subseteq \FK$, then, for all groups with $\cd_{\O_\FC} G \geq 3,$ we have
$$\cd_{\O_\FC} G = \gd_{\FC} G.$$
\end{prop}
\begin{proof} 
This follows directly from \cite[Theorems 0.1 and~0.2]{lueck-meintrup} and \cite[Lemma 3.5]{luecksurvey}.
\end{proof}

\begin{cor}\label{change-fam-Bredon}
Let $G$ be a t.d.l.c.~group and let $\FC$ be a family of closed subgroups such that $\FCO \subseteq \FC \subseteq \mathfrak{K}$. 
If $\cd_{\OCO} G \geq 3$  then
$\cd_{\O_{\FC}} G \leq \cd_{\OCO } G.$
\end{cor}
\begin{proof}  Let $X$ be a model for $E_{\FCO}G$ of minimal dimension. As $\FCO\subseteq\FC$, $X$ is a model for $E_\FC G$ as well in the case that $X^H$ is contractible for every $H\in \FC$. This last property follows by the proof of Lemma~3.5 in~\cite{luecksurvey}. Hence, $\cd_{\O_{\FC}} G  \leq\gd_{\FC} G\leq  \gd_{\FCO} G=\cd_{\OCO } G$
by Proposition~\ref{cd-gd-prop}. 
\end{proof}
Let $\mu$ be a left-invariant Haar measure of $G$. If $\sup\{\mu(K)\mid K\in \FCO\}<\infty$ the t.d.l.c.~group $G$ is said to be {\em $CO$-bounded}.
\begin{cor} Let $G$ be a non-compact t.d.l.c.~group. If $\gd_{\FCO} G=1$ then $\cd_{\O_\FCO} G=1$.
In addition, if $G$ is compactly generated and $CO$-bounded, then the reverse implication holds.
\end{cor}
\begin{proof}  As $G$ is non-compact, $\cd_{\OCO} G\neq 0.$ Hence, by Proposition~\ref{easy-prop}, the first part of the claim is clear.  
Assume now that $\cd_{\O_\FCO} G=1$. By Proposition~\ref{cd-rational-Bredon} and Remark~\ref{rem:rdcd}(a), $\cd_{\Q} G=1$. Assume that $G$ is compactly generated and $CO$-bounded. By~\cite[Theorem~B]{CMW}, $G$ is isomorphic to the fundamental group of a finite graph of profinite groups. The associated Bass-Serre tree is a one-dimensional model for $E_{\FCO}G$ and $\gd_{\FCO} G=1$.
\end{proof}

\begin{rem}
For an arbitrary non-discrete t.d.l.c.~group $G$ it is still unknown whether $\cd_{\Q} G=1$ implies a proper continuous action on a tree.
\end{rem}

The following example emerged during a discussion of the first author with Bianca Marchionna and Thomas Weigel. This is intended as an aside observation, included for readers who may find it of interest (see~\cite{ab} for the definition of  building and related notions).
\begin{prop} Let  $\Delta$ be a locally finite building of type $(W,S)$. Suppose that $G$ acts properly and cocompactly on $\Delta$ by type-preserving automorphisms. Then $\gd_{\FCO} G=\cd_{\OCO} G=\cd_\Q G$.
\end{prop}

\begin{proof} Recall that $\cd_\Q G\leq \cd_{\OCO} G\leq  \gd_{\FCO} G $ by Propositions~\ref{cd-rational-Bredon} and~\ref{easy-prop}. Let $|\Delta|$ denote the Bestvina realisation of the building $\Delta$ (cf.~\cite[Remark~10.4]{davis}). Its dimension is equal to $\cd_\Q(W)$, which also coincides with $\cd_\Q(G)$~\cite[Equation~(1.2)]{CMWarxiv}.
As $|\Delta|$ is a model for $E_{\FCO}G$, the claim follows.
\end{proof}


\end{document}